\documentclass[12pt]{amsart}
\usepackage{amsfonts,amssymb}
\usepackage{amsmath,amscd}
\usepackage[all]{xy}
\usepackage{youngtab}

\newtheorem{thm}{Theorem}
 \newtheorem{cor}{Corollary}
 
 \newtheorem{prop}{Proposition}
 \theoremstyle{definition}
 \newtheorem{defn}{Definition}
 \newtheorem{ex}{Example}
 \newtheorem{rem}{Remark}
\newcommand{\Q}{\mathbb{Q}}
\newcommand{\B}{\mathcal{B}}

\title{Schur-Weyl duality for $U_{v,t}(sl_{n})$}

\author{Yanmin Yang, Haitao Ma, Zhu-Jun Zheng}
\address{Department of Mathematics, Guangzhou University,
Waihuanxi Lu, Guangzhou Higher Education Mega Center, Panyu District, Guangzhou, P.R.China}
\email{yangym929@gmail.com}

\address{Department of Mathematics, South China University of Technology, Wushan Road, Tianhe District, Guangzhou, P.R.China}
\email{
mahaitao871219@163.com}

\address{Department of Mathematics, South China University of Technology, Wushan Road, Tianhe District, Guangzhou, P.R.China}
\email{
Zhengzj@scut.edu.cn
}

\date{\today}

\begin{document}
\begin{abstract}
In \cite{fl}, the authors get a new presentation of two-parameter quantum algebra $U_{v,t}(\mathfrak{g})$. Their presentation can cover all Kac-Moody cases. In this paper, we construct a suitable Hopf pairing such that $U_{v,t}(sl_{n})$ can be realized as Drinfeld double of certain Hopf subalgebras with respect to the Hopf pairing.
Using Hopf pairing, we  construct a $R$-matrix for $U_{v,t}(sl_{n})$ which will be used to
give the Schur-Weyl dual between $U_{v,t}(sl_{n})$ and  Hecke algebra $H_{k}(v,t)$. Furthermore,
 using the Fusion procedure we construct the primitive orthogonal idempotents of $H_{k}(v,t)$. As a corollary, we give the explicit construction of irreducible $U_{v,t}(sl_{n})$-representations of $V^{\otimes k}$.
\end{abstract}
\maketitle

\section{Introduction}
Classical Schur-Weyl duality  related irreducible finite-dimensional representations of the general linear and symmetric groups \cite{Weyl}. The quantum version for the quantum enveloping algebra $U_q(\mathfrak{sl}_n)$ and the Hecke algebra $H_q({\mathfrak S}_m)$ has
been one of the pioneering examples \cite{Jb2} in the fervent development of
quantum groups.  Two-parameter general linear  quantum groups were introduced by Takeuchi in 1990 \cite{Takeuchi}. The related references are \cite{AST, CM, DP, DT, JingLiu, Takeuchi}.
In 2001, Benkart and Witherspoon obtained the structure of two-parameter quantum groups corresponding to the general linear Lie algebra $gl_n$ and  the special linear Lie algebra $sl_n$ with a different motivation \cite{BW}. They showed that the quantum groups can be realized as
Drinfeld doubles of certain Hopf subalgebras with respect to Hopf pairings.
Using Hopf pairing, Benkarat and Witherspoon constructed $R$-matrix which is used to establish an analogue of Schur-Weyl duality \cite{benkart2001rep}.

In \cite{fl},
using geometric construction,  the authors  got a new presentation of generators and relations for a two-parameter quantum algebra $U_{v,t}$ determined by a certain matrix which may served as a generalized Cartan matrix. The two parameters $v$ and $t$ they used  are different from the one $(r, s)$.
Furthermore, their presentation covered all Kac-Moody cases, unlike the one in literature which mainly studies finite type and some affine types. A two-parameter quantum algebra $U_{v,t}$ is a two-cocycle deformation, depending only on the second parameter $t$, of its one-parameter analogue. And the algebra $U_{r,s}$ in \cite{BW}, is a two-cocycle deformation depending both parameters ($r=s^{-1}$).

 We focus on the two-parameter quantum group $U_{v,t}(sl_{n})$ for the purpose of  giving the Shur-Weyl dual between $U_{v,t}(sl_{n})$ and $H_{k}(v,t)$. We will show that
$U_{v,t}(sl_{n})$ also has a Drinfeld double realization by two certain Hopf subalgebras, but the $R$-matrix constructed by similar way in \cite{benkart2001rep} can not afford a representation of $H_k(v,t)$, we need to take a suitable modification.

This paper is organized as follows. In section \ref{sec2}, we give a Hopf pairing between two certain Hopf subalgebras of $U_{v,t}(sl_{n})$.
Then we prove that $U_{v,t}(sl_{n})$ can be realized as the Drinfeld double of certain Hopf subalgebras with respect to the Hopf pairing. In section \ref{sec3-tensor rep}, we construct the tensor power representation $V^{\otimes k}$ of $U_{v,t}(sl_{n})$ and $R$-matrix $\tilde{R}$. In section \ref{sec hecke alg}, we prove that the $U_{v,t}(sl_{n})$-module $V^{\otimes k}$ affords a representation of Hecke algebra $H_k(v,t)$. This leads to a Schur-Weyl duality between $U_{v,t}(sl_{n})$ and  Hecke algebra $H_{k}(v,t)$. In section \ref{sec idempotent}, we give a family of primitive orthogonal idempotents of $H_{k}(v,t)$. In section \ref{sec fusion}, irreducible representations of $U_{v,t}(sl_{n})$ are constructed by using the fusion procedure.

\section{Two-Parameter Quantum Group $U_{v,t}(sl_{n})$ And Its Drinfeld Double}\label{sec2}

In this section, we review the definition of two-parameter quantum algebra $U_{v,t}(sl_{n})$ introduced by Fan and Li in \cite{fl} and its Hopf algebra structure. In particular, $U_{v,t}(sl_{n})$ also can be realized as a Drinfeld double of its certain subalgebras.

\subsection{Two-Parameter Quantum Group $U_{v,t}(sl_{n})$}
In this paper,  we fix the  Cartan datum $(\Omega_{n-1}, \cdot)$  of type $A_{n-1}$, where
 $$\Omega_{n-1} =(\Omega_{ij})= \left(
                             \begin{array}{cccccc}
                               1 & 0 & 0 & \dots & 0 & 0 \\
                               -1 & 1 & 0 & \dots & 0 & 0 \\
                               0 & -1 & 1 & \dots & 0 & 0 \\
                               \vdots & \vdots & \vdots & \vdots & \vdots & \vdots \\
                               0 & 0 & 0 & \dots & 1 & 0 \\
                               0 & 0 & 0 & \dots & -1 & 1 \\
                             \end{array}
                           \right)_{(n-1) \times (n-1)},
$$
and for any $1\leq i, j\leq n-1$, denote $\langle i,j  \rangle=\Omega_{ij}$, then $i\cdot j$ can be defined as $\langle i, j\rangle +\langle j,i\rangle$.

\begin{defn}[\cite{fl}]
 The two-parameter quantum algebra $U_{v,t}(sl_{n})$ associated to $\Omega_{n - 1}$ is an associative $\mathbb{Q}(v,t)$-algebra with 1 generated by symbols $E_i, F_i,$ $K_i^{\pm 1}, {K_i'}^{\pm 1}$, $\forall 1\leq i < n$   and subject to the following relations.
\begin{eqnarray*}
  (R1)& & K_i^{\pm 1}K^{\pm 1}_j=K^{\pm 1}_jK_i^{\pm 1},\ \ {K_i'}^{\pm 1}{K_j'}^{\pm 1}={K_j'}^{\pm 1}{K_i'}^{\pm 1},\\
     & & K_i^{\pm 1}{K_j'}^{\pm 1}={K_j'}^{\pm 1}K_i^{\pm 1},\ \ K_i^{\pm 1}K_i^{\mp 1}=1={K_i'}^{\pm 1}{K_i'}^{\mp 1},\\
  (R2)& &K_iE_jK^{-1}_i=v^{i\cdot j}t^{\langle i,j\rangle - \langle j,i\rangle}E_j,\ \ {K_i'}iE_j{K_i'}^{-1}=v^{-i\cdot j}t^{\langle i,j\rangle - \langle j,i\rangle}E_j,\\
     & &K_iF_jK^{-1}_i=v^{- i\cdot j}t^{\langle j,i\rangle - \langle i,j\rangle}F_j,\ \ {K_i'}F_j{K_i'}^{-1}=v^{i\cdot j}t^{\langle j,i\rangle - \langle i,j\rangle}F_j,\\
  (R3)& & E_iF_j-F_j E_i=\delta_{ij}\frac{K_{i}-{K_i'}}{v-v^{-1}},\\
  (R4)& & [E_i, E_j]=[F_i, F_j]=0 \ if \mid i-j \mid >1, \\
  (R5)& & E_i^2E_{i+1}-t(v+v^{-1})E_iE_{i+1}E_i+t^2 E_{i+1}E_{i}^2=0,\\
  & & E_iE_{i+1}^2-t(v+v^{-1})E_{i+1}E_{i}E_{i+1}+t^2 E_{i+1}^2E_{i}=0,\\
  (R6)& & F_i^2F_{i+1}-t^{-1}(v+v^{-1})F_iF_{i+1}F_i+t^{-2} F_{i+1}F_{i}^2=0, \\
  & &  F_iF_{i+1}^2-t^{-1}(v+v^{-1})F_{i+1}F_{i}F_{i+1}+t^{-2} F_{i+1}^2F_{i}=0.\\
\end{eqnarray*}
\end{defn}

 There is a Hopf algebra structure on the algebra $U_{v,t}(sl_n)$ with the comultiplication $\Delta$  by
  $$\begin{array}{llll}
   &\Delta(K_i^{\pm 1})=K_i^{\pm 1} \otimes K_i^{\pm 1},&\Delta({K_i'}^{\pm 1})={K_i'}^{\pm 1}\otimes {K_i'}, &\vspace{4pt}\\
   &\Delta(E_i)=E_i\otimes 1+K_i\otimes E_i,& \Delta(F_i)=1 \otimes F_i+F_i\otimes K_i', & \vspace{4pt}\\
\end{array}$$
the counit $\varepsilon$ by
 $$\begin{array}{llll}
  &\varepsilon(K_i^{\pm 1})=\varepsilon({K_i'}^{\pm 1})=1,& \varepsilon(E_i)=\varepsilon(F_i)=0,& S(K_i^{\pm 1})=K_i^{\mp 1},\vspace{4pt}\\
 \end{array}$$
and the antipode $S$ by
$$\begin{array}{llll}
  & S({K_i'}^{\pm 1})={K_i'}^{\mp 1},&
  S(E_i)=-K_i^{-1}E_i,& S(F_i)=-F_i{K_i'}^{-1}.
  \end{array}$$

\subsection{Drinfeld Double Realization Of  $U_{v,t}(sl_{n})$}

\begin{defn} (See \cite{Jos}, 3.2.1)\label{def pairing}
A Hopf paring of two Hopf algebras $H$ and $H'$ is a bilinear form $(-,-): H' \times H \longrightarrow K \ (a\ field)$ such that
\begin{itemize}
  \item [(1).] $(1,h)=\epsilon_H(h)$, \qquad $( h',1)=\epsilon_{H'}(h')$;
  \item [(2).] $( h',hk)=(\Delta_{H'}(h'), h\otimes k)=\sum( h'_{(1)}, h)( h'_{(2)}, k)$;
  \item [(3).] $( h'k',h)=( h'\otimes k',\Delta_{H}(h))=\sum( h', h_{(1)})( k', h_{(2)})$;
\end{itemize}
for all $h,k\in H$, $h',k'\in H'$, where $\epsilon_H$ and $\epsilon_{H'}$ are the counits of $H$ and $H'$ respectively, and $\Delta_{H}$ and $\Delta_{H'}$ are their comultiplications.
For $h\in H$, $\Delta(h)=\sum h_{(1)}\otimes h_{(2)}$.
\end{defn}

A direct consequence is that
\begin{equation}\label{}
  ( S_{H'}(h'), h)=( h', S_H(h))
\end{equation}
for all $h\in H$ and $h'\in H'$, where $S_{H'}$ and $S_H$ are the antipodes of $H'$ and $H$ respectively.

Let $\B$ (resp. $\B'$) be the Hopf subalgebra of $U_{v,t}(sl_{n})$ generated by $E_i$, $K_i^{\pm1}$ (resp. $F_i, K_i^{'\pm1}$) for $1\leq i <n$. $\B^{'coop}$ is the Hopf algebra having the opposite comultiplication to the Hopf algebra $\B'$ and $S_{\B^{'coop}}=S_{\B'}^{-1}$, $\Delta_{\B^{'coop}}=\Delta^{op}$.

\begin{prop}
There exists a unique Hopf pairing $(-,-): \B^{'coop} \times \B \longrightarrow \Q(v,t)$ such that
\begin{eqnarray}\label{hopfpairing1}
  (F_i, E_j) & = & \frac{\delta_{ij}}{v^{-1}-v}, \\ \label{hopfpairing2}
 (K_i', K_j) & = & v^{j\cdot i}t^{\langle j,i \rangle-\langle i, j\rangle},
\end{eqnarray}
for any $1\leq i, j <n$,
and all other pairs of generators are 0. Moreover, we have
\begin{equation}\label{}
  ( S(a), S(b))=( a, b)
\end{equation}
for any $a\in \B^{'coop}$, $b\in \B$.
\end{prop}
\begin{proof}
Any Hopf pairing of bialgebras is determined by the values on the generators, so the uniqueness is clear.  The process of proof reduces to the existence.

The pairings defined by (\ref{hopfpairing1}) and (\ref{hopfpairing2}) in the proposition can be extended to a bilinear form on  $\B^{'coop} \times \B$ by requiring that the conditions (1), (2) and (3) in definition \ref{def pairing} hold. We only need to verify that the relations (\ref{hopfpairing1}) and (\ref{hopfpairing2}) in $\B^{'}$ and $\B$ are preserved.

It is straightforward to check that the bilinear form preserves all the relations among the $K_i^{\pm1}$ in $\B$ and the $K_i^{'\pm1}$ in $\B'$.
Next, for any $1\leq i, j<n$, we check
  $$( X, K_jE_i)=( X, v^{j\cdot i}t^{\langle j,i \rangle - \langle i, j \rangle}E_iK_j),$$
where $X$ is any word in the $F_i$ and $K_i^{'\pm1}$, $1\leq i<n$.
If $X=K'_kF_l$, the left hand side

\begin{align*}
( X, K_jE_i)&=( \Delta^{op}(K'_k)\Delta^{op}(F_l),  K_j\otimes E_i) \\
    & =( K'_kF_l\otimes K'_k+ K'_kK'_l \otimes K'_kF_l, K_j\otimes E_i ) \\
       &=( K'_kF_l, K_j)( K'_k, E_i)+( K'_kK'_l, K_j)( K'_kF_l, E_i) \\
&= \frac{1}{v^{-1}-v}v^{j\cdot i}t^{\langle j, i\rangle -\langle i, j \rangle}( K'_k, K_j )( K'_k, K_i ) \qquad       (l=i)
\end{align*}
the right hand side
\begin{align*}
   ( X, v^{j\cdot i}t^{\langle j,i \rangle - \langle i, j \rangle}E_iK_j)&=
v^{j\cdot i}t^{\langle j,i \rangle - \langle i, j \rangle}( \Delta^{op}(K'_kF_l),  E_i\otimes K_j)\\
& = v^{j\cdot i}t^{\langle j,i \rangle - \langle i, j \rangle} ( K'_kF_l\otimes K'_k+ K'_kK'_l \otimes K'_kF_l,  E_i\otimes K_j ) \\
& = \frac{1}{v^{-1}-v}v^{j\cdot i}t^{\langle j, i\rangle -\langle i, j \rangle}( K'_k, K_j )( K'_k, K_i ) \qquad       (l=i)
\end{align*}
Hence, $( X, K_jE_i)=( X, v^{j\cdot i}t^{\langle j,i \rangle - \langle i, j \rangle}E_iK_j).$

In particular, it can be similarly checked that the bilinear form preserves all the other relations in $\B$ and $\B'$.

\end{proof}

\begin{defn}(See \cite{Jos}, 3.2)
If there is a Hopf pairing between Hopf algebras $H$ and $H'$, then we may form the Drinfeld double $D(H, H^{'coop})$, where $ H^{'coop}$ is the Hopf algebra having the opposite coproduct to $H$. $D(H, H^{'coop})$ is a Hopf algebra whose underlying vector space is $H\otimes H'$ with the tensor product coalgebra structure. The algebra structure is given by as follows:
$$(a\otimes f)(a'\otimes f')=\sum( S_{H^{'coop}}(f_{(1)}), a'_{(1)})( f_{(3)}, a'_{(3)}) aa'_{(2)}\otimes f_{(2)}f'$$
for $a, a'\in H$ and $f, f'\in H'$. And the antipode $S$ is given by
$$S(a\otimes f)=(1\otimes S_{H^{'coop}}(f))(S_H(a)\otimes 1).$$
\end{defn}
Clearly, the algebras $H$ and $H^{'coop}$ are identified with $H\otimes 1$ and $1\otimes H^{'coop}$ respectively in $D(H, H^{'coop})$.

\begin{prop}
$D(\B, \B^{'coop})$ is isomorphic to $U_{v,t}(sl_n)$.
\end{prop}
\begin{proof}
Define the embedding maps
\[
\begin{array}{cccl}
 \iota: & \B & \longrightarrow &   D(\B, \B^{'coop})  \\
   &  E_i & \mapsto &  \widehat{E_i}:=\iota(E_i)=E_i \otimes 1 \\
   & K_i^{\pm1}  & \mapsto &  \widehat{K_i}^{\pm1}:=\iota(K_i^{\pm1})=K_i^{\pm1} \otimes 1, \\
\end{array}
\]
and
\[
\begin{array}{cccl}
 \iota': & \B^{'coop} & \longrightarrow &   D(\B, \B^{'coop})  \\
   &  F_i & \mapsto &  \widehat{F_i}:=\iota(F_i)=1 \otimes F_i\\
   & K_i^{'\pm1} & \mapsto &   \widehat{K_i}^{'\pm1}:=\iota(K_i^{'\pm1})=1 \otimes K_i^{'\pm1}. \\
\end{array}
\]
Then $\B$ and $\B^{'coop}$ can be viewed as subalgebras in $D(\B, \B^{'coop})$.
A map $\varphi$ between  $D(\B, \B^{'coop})$ and $U_{v,t}(sl_n)$ is defined as follows:
\[
\begin{array}{cccl}
  \varphi: & D(\B, \B^{'coop}) & \longrightarrow &   U_{v,t}(sl_n)  \\
   &  \widehat{E_i} & \longmapsto &  \varphi(\widehat{E_i})=E_i, \\
   & \widehat{F_i}  & \longmapsto &  \varphi(\widehat{F_i})=F_i, \\
   & \widehat{K_i}^{\pm1} &  \longmapsto &  \varphi(\widehat{K_i}^{\pm1})=K_i^{\pm1},  \\
   & \widehat{K_i}^{'\pm1} & \longmapsto &  \varphi(\widehat{K_i}^{'\pm1})=K_i^{'\pm1}.
\end{array}
\]

Note that, $\varphi$ preserves the coalgebra structures. Next we will check that $\varphi$ preserves the relations (R1-R6) in $U_{v,t}(sl_n)$.
Consider $\varphi(\widehat{K_j}\widehat{K_i}'-\widehat{K_i}'\widehat{K_j})$. By definition,
$$\widehat{K_j}\widehat{K_i}'=(K_j\otimes 1)(1\otimes K'_i)=K_j\otimes K'_i.$$
To calculate $\widehat{K_i}'\widehat{K_j}$, we have
$$\Delta^2({K_j})=K_j\otimes K_j\otimes K_j, \quad (\Delta^{op})^2({K'_i})=K'_i\otimes K'_i\otimes K'_i,$$
so that
\begin{align*}
 \widehat{K_i}'\widehat{K_j} & =(1\otimes K'_i)(K_j\otimes 1) \\
   & = ( S_{\B^{'coop}}(K'_i), K_j)( K'_i, K_j) K_j\otimes K'_i \\
  & = K_j\otimes K'_i.
\end{align*}
That is, $\varphi(\widehat{K_j}\widehat{K_i}'-\widehat{K_i}'\widehat{K_j})=0=K_jK_i'-K_i'K_j$.
Other relations in $\B$ (resp. in $\B'$) also can be verified  by the same way. We verify  the mixed relation R3.

Similarly, $\widehat{E_i}\widehat{F_j}=(E_i\otimes 1)(1\otimes F_j)=E_i\otimes F_j.$
In order to calculate  $\widehat{F_j}\widehat{E_i}$, we use
\begin{eqnarray*}
  \Delta^2(E_i) &=& E_i\otimes 1\otimes 1+K_i\otimes E_i\otimes 1+K_i\otimes K_i\otimes E_i, \\
  (\Delta^{op})^2(F_j) &=& 1\otimes 1\otimes F_j+1\otimes F_j\otimes K'_j+F_j\otimes K'_j\otimes K'_j,
\end{eqnarray*}
so that
\begin{align*}
    \widehat{F_j}\widehat{E_i}& =(1\otimes F_j)(E_i\otimes 1) \\
& =  ( S_{\B^{'coop}}(F_j), E_i)( K'_j, 1) 1\otimes K'_j+ ( 1,K_i)( K'_j, 1) E_i\otimes F_j+( 1,K_i)( F_j, E_i) K_i\otimes 1\\
   & = ( -F_j(K'_j)^{-1}, E_i)( K'_j, 1)  \widehat{K_j}'+ ( 1,K_i)( K'_j, 1) \widehat{E_i}\widehat{F_j}+( 1,K_i)( F_j, E_i) \widehat{K_i}\\
  & =  -\frac{\delta_{ij}}{v^{-1}-v}\widehat{K_j}'+\widehat{E_i}\widehat{F_j}+\frac{\delta_{ij}}{v^{-1}-v}\widehat{K_i}.
\end{align*}
That is, $[\widehat{E_i}, \widehat{F_j}]=\frac{\delta_{ij}}{v^{-1}-v}(\widehat{K_j}'-\widehat{K_i})$.
Applying $\varphi$ gives the desired relation (R3) in $U_{v,t}(sl_n)$.
\end{proof}

Let $U^0$ denote the subalgebra of $U_{v,t}(sl_n)$ generated by $K_i^{\pm1}$, $K_i^{'\pm1}$, $1\leq i< n$. Let $U^+$ (resp. $U^-$) denote the subalgebra of $\B$ (resp. $\B'$) generated by $E_i$ (resp. $F_i$), $1\leq i< n$. Then we have $$D(\B, \B^{'}) \cong U^+ \otimes U^0 \otimes U^-.$$

\begin{cor}
The algebra $U_{v,t}(sl_n)$ has a triangular decomposition
$$U_{v,t}(sl_n) \cong U^- \otimes U^0 \otimes U^+.$$

\end{cor}

\section{Finite Dimensional Representations Of $U_{v,t}(sl_{n})$ }\label{sec3-tensor rep}

\subsection{The Natural Representation Of $U_{v,t}(sl_{n})$ }

Set $\Lambda=\mathbb{Z}\epsilon_1\oplus
\mathbb{Z}\epsilon_2\oplus\cdots\oplus \mathbb{Z}\epsilon_n$. For any $\lambda=\sum\limits_{j=1}^n \lambda_j \epsilon_j \in \Lambda$, one defines the algebra homomorphism $\widehat{\lambda}: U^0 \rightarrow \Q(v,t)$ by
\begin{equation}\label{homo}
  \widehat{\lambda}(K_i)=v^{\sum\limits_{j=1}^n\lambda_j i\cdot j} t^{\sum\limits_{j=1}^n\lambda_j (\langle i,j\rangle-\langle j,i\rangle)-1}, \qquad
 \widehat{\lambda}(K_i')=v^{-\sum\limits_{j=1}^n\lambda_j i\cdot j} t^{\sum\limits_{j=1}^n\lambda_j (\langle i,j\rangle-\langle j,i\rangle)-1},
\end{equation}
where $\langle i, n \rangle= \delta_{in}$, and
$\langle n, i \rangle=\left\{
   \begin{array}{ll}
     0, & \hbox{$1\leq i< n-1$;} \\
     -1, & \hbox{$i=n-1$;} \\
     1, & \hbox{$i=n$;}
   \end{array}
 \right.$
for $1 \leq i \leq n$.

Let $V_n$ be the $n$-dimensional $\Q(v,t)$ vector space with basis $\{v_1,v_2,\cdots,v_n\}$.
For any $1\leq i, j \leq n$, set $E_{ij}$ be the $n\times n$ matrices with entry $1$ in row $i$ and column $j$ and other  entries $0$.

We define an  $U_{v,t}(sl_{n})$ representation $\rho_n': U_{v,t}(sl_{n}) \times V_n \rightarrow V_n$ by the following way:
\begin{align*}
  \rho_n'(E_i) = E_{i,i+1}, \qquad   \rho_n'(F_i) = E_{i+1,i},  \\
  \rho_n'(K_i) =t^{-1}I_n +  (v-t^{-1})E_{i,i} + (v^{-1}-t^{-1})E_{i+1,i+1}, \\
  \rho_n'(K_i') =t^{-1}I_n + (v^{-1}-t^{-1})E_{i,i} + (v-t^{-1})E_{i+1,i+1}.
\end{align*}
This  follows from the fact that $K_i v_j=v^{\sum\limits_{k=j}^n i\cdot k} t^{\sum\limits_{k=j}^n (\langle i,k\rangle-\langle k,i\rangle)-1}v_j$ for all $1\leq i\leq n-1,\ 1\leq j\leq n$ that $v_j$ corresponds to the weight $\sum\limits_{k=j}^n \epsilon_k$. Thus, $V_n=\bigoplus\limits_{j=1}^n V_{\sum\limits_{k=j}^n \epsilon_k}$ is the natural analogue of the $n$-dimensional representation of $sl_n$, and $(\rho_n',V_n)$ is an irreducible representation of $U_{v,t}(sl_n)$.

\subsection{Tensor Power representations Of $U_{v,t}(sl_{n})$}

\begin{defn}
Let $M, N$ be $U_{v,t}(sl_{n})$-modules. For $a \in U_{v,t}(sl_{n}) $, $ u \in M$ and $v \in N$,  we define $a(u \otimes v) = \Delta(a)(u \otimes v))$, then under such action, $M \otimes N$ is a $U_{v,t}(sl_{n})$-module. We call it the tensor power of modules $M$ and $N$.
\end{defn}

\begin{defn}
More generally, suppose that $M_1, M_2, \cdots M_n $ are $U_{v,t}(sl_{n})$-modules.
For $a \in U_{v,t}(sl_{n})$, $m_i \in M_i$, we define
$a(m_1\otimes m_2\otimes \cdots \otimes m_{n})=\triangle^{n-1}(a)(m_1\otimes m_2\otimes \cdots \otimes m_{n})$,
where $\Delta^{k} = (\Delta \otimes id^{\otimes k - 1})\Delta^{k - 1}$ for any $k\in \mathbb{N}$.
Then $M_1\otimes M_2\otimes \cdots \otimes M_{n}$ is a $U_{v,t}(sl_{n})$-module. We call it the tensor product of modules $M_1, M_2, \cdots ,M_{n}.$
\end{defn}

\begin{rem}
It is easy to know that for any $k\in \mathbb{N}$, $1 \leq j < n,$
$$\Delta^{k-1}(E_j)= \sum_{i=1}^{k}
\underbrace{K_j\otimes \cdots  \otimes K_j}_{i-1}\otimes E_j \otimes
\underbrace{1\otimes \cdots \otimes 1}_{k-i}.$$

$$\Delta^{k-1}(F_j) = \sum_{i=1}^{k}
\underbrace{1\otimes \cdots \otimes 1}_{k-i}\otimes F_j \otimes
\underbrace{K_j' \otimes \cdots \otimes K_j'}_{i-1}.$$

$$\Delta^{k-1}(K_i) = K_i \otimes K_i \otimes \cdots \otimes K_i.$$

$$\Delta^{k-1}(K_i') = K_i' \otimes K_i' \otimes \cdots \otimes K_i'.$$
\end{rem}

Generally,  the $k$-fold tensor power $(\rho_n, V_n^{\otimes k}) $ of $(\rho_n',V_n)$ is also a  $U_{v,t}(sl_n)$-module, where $V_n^{\otimes k}=V_n \otimes V_n \otimes \cdots \otimes V_n$ ($k$ factors).

\begin{prop}
For $n\geq k$, if $(\rho_n',V_n)$ is the natural representation of $U_{v,t}(sl_n)$, then $(\rho_n, V_n^{\otimes k}) $ is a cyclic $U_{v,t}(sl_n)$-module generated by $\{v_1,v_2,\cdots,v_n\}$.

\end{prop}

\begin{proof}
The proof is similarly to lemma 6.2 in \cite{benkart2001rep}. We omit it.

\end{proof}

\subsection{The $R$-Matrix}\label{sec.rmatrix}

To obtain a $H_k(v,t)$-representation from $U_{v,t}(sl_n)$-representation $(\rho_n, V_n^{\otimes k}) $,
we shall construct a $R$-matrix.

As above notation, denote by $U^+$ (resp. $U^-$) the subalgebra of $U_{v,t}(sl_n)$ generated by $1$ and $E_i$ (resp. $1$ and $F_i$) for all $1\leq i \leq n-1$.
Then  $U^+$ has a decomposition $U^+=\bigoplus_{\xi \in \Lambda^+}U_\xi^+$, where
\begin{equation}
  U_\xi^+=\{x\in U^+ \mid K_i x= v^{\sum\limits_{j=1}^n\xi_j i\cdot j} t^{\sum\limits_{j=1}^n\xi_j \langle i,j\rangle-\langle j,i\rangle} x K_i,\ 1\leq i \leq n-1\}
\end{equation}
for $\xi=\sum\limits_{j=1}^n\xi_j \epsilon_j$, $\xi_j\geq 0$.

It can be checked that $ U_\xi^+$ is spanned by all the monomials $E_{i_1} E_{i_2} \cdots E_{i_m}$ such that $\xi=\epsilon_{i_1} + \epsilon_{i_2} + \cdots \epsilon_{i_m}$.
For $U^-$ we have similar decomposition $U^-=\bigoplus_{\xi \in \Lambda^+}U_{-\xi}^-$.
Let $d_\xi$ be the $\Q(v, t)$ dimension of $U_\xi^+$. Assume $\{ u_k^\xi \}_{k=1}^{d_\xi}$ is a basis for
$U_\xi^+$, and $\{ v_k^\xi \}_{k=1}^{d_\xi}$ is the dual basis for $U_{-\xi}^-$ with respect to to the Hopf pairing defined by (\ref{hopfpairing1}) and (\ref{hopfpairing2}). That is to say, if $\{E_1\}$ is a basis for $U_{\epsilon_1}^+$, then $\{(v^{-1}-v)F_1\}$ is the dual basis for
$U_{-\epsilon_1}^-$.

If $\lambda=\sum\limits_{j=1}^n \lambda_j \epsilon_j \in \Lambda$, set
\begin{align*}
  K_\lambda=K_1^{\lambda_1}\cdots K_{n-1}^{\lambda_{n-1}}A_n^{\lambda_n}, \qquad
 K'_\lambda=(K'_1)^{\lambda_1}\cdots (K'_{n-1})^{\lambda_{n-1}}B_n^{\lambda_n},
\end{align*}
and
\begin{align*}
  \Theta = \sum_{\xi \in \Lambda^+} \sum_{k=1}^{d_\xi} v_k^\xi \otimes u_k^\xi.
\end{align*}

Let $\widetilde{R}=\widetilde{R}_{V_n,V_n}: V_n \otimes V_n \longrightarrow V_n \otimes V_n$ be the
$R$-matrix defined by
\begin{equation}\label{Rmatrix}
  \widetilde{R}_{V_n,V_n}(v_i \otimes v_j )=\Theta \circ f(v_j \otimes v_i )
\end{equation}
where $f(v_j \otimes v_i )=(K'_\lambda, K_\mu)^{-1}$ when $v_i\in V_\mu$ and $v_j\in V_\lambda$, the Hopf pairing $(-, -)$ is defined by (\ref{hopfpairing1}), (\ref{hopfpairing2}) and
\begin{align*}
  (B_n, A_n)=1, \qquad (B_n, K_j)=v^{\langle n, j \rangle}t^{-\langle n, j \rangle}, \qquad
(K'_i, A_n)=(vt)^{\langle n, i \rangle}.
\end{align*}

More precisely, $\widetilde{R}$ acts on $V_n \otimes V_n$ by
\[
 \widetilde{R} = \sum_{i < j}vt E_{j,i}\otimes E_{i,j} + \sum_{i<j}vt^{-1} E_{i,j}\otimes E_{j,i} + t^{-1}(1 - v^2)\sum_{i < j}E_{j,j}\otimes E_{i,i}
+ \sum_{i}E_{i,i}\otimes E_{i,i}.
\]

\begin{prop}\label{braidrel}
Let $\widetilde{R}_i$ be the $U_{v,t}(sl_n)$-module isomorphism on $V_n^{\otimes k}$ defined by
$$
\widetilde{R}_i(z_1\otimes z_2 \otimes \cdots \otimes z_k)=z_1\otimes z_2 \otimes \cdots \otimes z_{i-1}
\otimes \widetilde{R}(z_i \otimes z_{i+1}) \otimes z_{i+2} \otimes \cdots \otimes z_k,
$$
where $z_1, z_2, \cdots,  z_k \in V_n$. Then $\widetilde{R}_i$ satisfy the Yang-Baxter equations. That is to say, the following braid relations hold:
\[
  \begin{array}{ll}
\medskip
    \widetilde{R}_i \widetilde{R}_j=\widetilde{R}_j \widetilde{R}_i, & \hbox{$\mid i-j \mid \neq 1$;} \\
    \widetilde{R}_i \widetilde{R}_{i+1} \widetilde{R}_i=\widetilde{R}_{i+1} \widetilde{R}_i \widetilde{R}_{i+1}, & \hbox{$1\leq i < k$.}
  \end{array}
\]
\end{prop}

\section{Hecke Algebra and The Schur-Weyl Duality For $U_{v,t}(sl_n)$}\label{sec hecke alg}

Let $v, t$ be any formal variables. We introduce the Hecke algebra $H_k(v,t)$ as follows.

\begin{defn}\label{hecke}
The Hecke algebra $H_k(v,t)$ be the unital associate algebra over $\Q(v,t)$ with generators $T_{i}, 1 \leq i < k$, subject to the relations:

\begin{enumerate}
  \item[(H1)] $T_iT_{i+1}T_{i}=T_{i+1}T_{i}T_{i+1}$, \ $1\leq i <k$,
  \item[(H2)]  $T_iT_j = T_jT_i$, \ ${\rm if} \ |i - j| \geq 2$,
  \item[(H3)]  $(T_i - v^{-1}t)(T_i + vt) = 0$, \ $1\leq i <k$.
\end{enumerate}

\end{defn}

The $R$-matrix $\widetilde{R}$ defined in section \ref{sec.rmatrix} only satisfies the braid relations. In order to construct an action for two-parameter Hecke algebra $H_k(v,t)$ on $V_n^{\otimes k}$, we must modify the $R$-matrix $\widetilde{R}$.
Set
\begin{align}\label{Rmatix}
  R= \sum_{i < j}t^2 E_{j,i}\otimes E_{i,j} + \sum_{i<j} E_{i,j}\otimes E_{j,i} + (v^{-1} - v)t\sum_{i < j}E_{j,j}\otimes E_{i,i}
+v^{-1}t \sum_{i}E_{i,i}\otimes E_{i,i}.
\end{align}
Let $R_i$ be the action on  $V_n^{\otimes k}$ defined by
$$
R_i(z_1\otimes z_2 \otimes \cdots \otimes z_k)=z_1\otimes z_2 \otimes \cdots \otimes z_{i-1}
\otimes R(z_i \otimes z_{i+1}) \otimes z_{i+2} \otimes \cdots \otimes z_k,
$$
for any $z_i \in V_n$. Furthermore, $R_i$ is an $U_{v,t}(sl_n)$-module isomorphism on $V_n^{\otimes k}$.

\begin{prop}
As above notations, if we let $\delta_n(T_i) = R_i$, then $(\delta_n, V_n^{\otimes k})$ is a representation of $H_k(v,t)$. That is to say, the $U_{v,t}(sl_n)$-representation $V_n^{\otimes k}$ affords a representation of Hecke algebra $H_k(v,t)$.
\end{prop}

\begin{proof}
The braid relations and  the commutativity of non-adjacent reflections follow from proposition \ref{braidrel}. We only need to check the relation $(H3)$ in definition \ref{hecke}.
For any subset $\{l_1, l_2, \cdots, l_k  \} \subseteq [1, n]$, if $l_i>l_{i+1}$, then
\begin{align*}
  (R_i)^2 (v_{l_1} \otimes \cdots \otimes v_{l_k}) & = R_i(v_{l_1} \otimes \cdots \otimes v_{l_{i-1}}\otimes ( v_{l_{i+1}}\otimes v_{l_i}+(v^{-1}-v)tv_{l_{i}}\otimes v_{l_i+1} ) \otimes \cdots \otimes v_{l_k}) \\
   & =t^2(v_{l_1} \otimes \cdots \otimes v_{l_k})+(v^{-1}-v)t R_i(v_{l_1} \otimes \cdots \otimes v_{l_k}) \\
& = (t^2+(v^{-1}-v)t R_i)(v_{l_1} \otimes \cdots \otimes v_{l_k}).
\end{align*}
If $l_i=l_{i+1}$, we have
\begin{align*}
(t^2+(v^{-1}-v)t R_i)(v_{l_1} \otimes \cdots \otimes v_{l_k}) &= (t^2+(v^{-1}-v)v^{-1}t^2 )(v_{l_1} \otimes \cdots \otimes v_{l_k})\\
& = v^{-2}t^2(v_{l_1} \otimes \cdots \otimes v_{l_k}) \\
& = (R_i)^2 (v_{l_1} \otimes \cdots \otimes v_{l_k}).
\end{align*}
For the last case $l_i<l_{i+1}$, we have
\begin{align*}
(R_i)^2 (v_{l_1} \otimes \cdots \otimes v_{l_k}) & =t^2 R_i(v_{l_1} \otimes \cdots \otimes v_{l_{i-1}}\otimes  v_{l_{i+1}}\otimes v_{l_i}  \otimes \cdots \otimes v_{l_k}) \\
& =t^2(v_{l_1} \otimes \cdots \otimes v_{l_k})+(v^{-1}-v)t^3(v_{l_1} \otimes \cdots \otimes v_{l_{i-1}}\otimes  v_{l_{i+1}}\otimes v_{l_i}  \otimes \cdots \otimes v_{l_k}) \\
& =(t^2+(v^{-1}-v)t R_i)(v_{l_1} \otimes \cdots \otimes v_{l_k}).
\end{align*}

\end{proof}

This leads to the Schur-Weyl duality between $U_{v,t}(sl_n)$ and $H_k(v,t)$.

\begin{thm} Assume $v^2$ is not a root of unity. Then
\begin{enumerate}
  \item $\delta_n(H_k(v,t)) = End_{U_{v,t}(sl_n)}(V_n^{\otimes k});$
  \item for $n \geq k$, we have $End_{U_{v,t}(sl_n)}(V_n^{\otimes k})\cong H_k(v,t).$
\end{enumerate}
\end{thm}

\begin{proof}
For conclusion (1), the proof is similarly to \cite{BW}. We consider the conclusion (2).
Assume $f\in End_{U_{v,t}(sl_n)}(V_n^{\otimes k})$, $\underline{v}=v_1\otimes v_2\otimes \cdots \otimes v_k \in V^{\otimes k}$, then $f(\underline{v})$ must be the linear combinations of $v_{\sigma(1)}\otimes v_{\sigma(2)}\otimes \cdots \otimes v_{\sigma(k)}$ for some $\sigma \in \mathfrak{S}_k$.
We will show that there is an element $T^\sigma \in H_k(v,t)$ such that $R^\sigma := \delta_n(T^\sigma)$ in $End_{U_{v,t}(sl_n)}(V_n^{\otimes k})$, $R^\sigma(\underline{v})=v_{\sigma(1)}\otimes v_{\sigma(2)}\otimes \cdots \otimes v_{\sigma(k)}$.
For any element $\sigma$ in $\mathfrak{S}_k$, $\sigma$ can be written as a product of transpositions, denoted by $\sigma=\tau_{i_1}\cdots \tau_{i_m}$, where $\tau_{i_l}=(i_l i_l+1)$.
For distinct index $j_i, \cdots, j_k$, we set
\begin{align*}
  R^{\tau_{i_l}}(v_{j_1}\otimes \cdots \otimes v_{j_k}) & = \left\{
                                               \begin{array}{ll}
                                                 ((v-v^{-1})t Id+R_{i_l})(v_{j_1}\otimes \cdots \otimes v_{j_k}), & \hbox{if $j_{i_l}> j_{i_l+1}$;} \\
                                                 t^{-2}R_{i_l}(v_{j_1}\otimes  \cdots \otimes v_{j_k}), & \hbox{if $j_{i_l}< j_{i_l+1}$.}
                                               \end{array}
                                             \right.
 \\
\end{align*}
Then defining $R^\sigma=R^{\tau_{i_m}}\circ \cdots \circ R^{\tau_{i_1}}$ in $End_{U_{v,t}(sl_n)}(V_n^{\otimes k})$, it can be checked that
$R^\sigma(\underline{v})=v_{\sigma(1)}\otimes v_{\sigma(2)}\otimes \cdots \otimes v_{\sigma(k)}$. Therefore, the map $\delta_n: H_k(v,t) \rightarrow End_{U_{v,t}(sl_n)}(V_n^{\otimes k})$ is a surjective, and
\begin{align*}
  End_{U_{v,t}(sl_n)}(V_n^{\otimes k}) & =span_{\Q(v,t)}\{ R^{\sigma} | \sigma \in \mathfrak{S}_k \}.
\end{align*}
Consequently, $dim_{\Q(v,t)}End_{U_{v,t}(sl_n)}(V_n^{\otimes k})=k!=dim_{\Q(v,t)}H_k(v,t)$.
It follows that (2) holds.

\end{proof}

\begin{cor}\label{decom}
Assume $v^2$ is not a root of unity. The space $V^{\otimes m}$ as an $U_{v,t}(sl_n)\otimes H_k(v,t)$-module has the decomposition
\begin{equation*}
V^{\otimes k}\cong \bigoplus_{\lambda} V_{\lambda}\otimes V^{\lambda},
\end{equation*}
where the partition $\lambda$ of $k$ runs over the set of partitions such that
$l(\lambda)\leq n$,
$V_{\lambda}$ is the $U_{v,t}(sl_n)$-module associated to $\lambda$,
$V^{\lambda}$ is the $H_k(v,t)$-module corresponding to $\lambda$.
\end{cor}

\section{The Primitive Orthogonal Idempotents Of  $H_k(v,t)$}\label{sec idempotent}

For any $i=1, 2, k-1$, let $s_i=(i, i+1)$ be the transposition in the symmetric group $S_k$. Choose a reduced decomposition $w=s_{i_1}\cdots s_{i_l}$ for $w\in S_k$, denote $T_{w}=T_{i_1}\cdots T_{i_l}$. Then
$T_{w}$ does not depend on the reduced decomposition,
and the set $\{T_w\ |\ w\in S_k\}$ is a basis of
$H_k(v,t)$ over $\Q(v, t)$.

The Jucys-Murphy elements
$y_1, \cdots, y_k$ of $H_k(v,t)$ are defined inductively by
\begin{gather}
y_1 = 1,\qquad
y_{i+1} =t^{-2} T_i y_i T_i   \qquad\text{for}\quad i=1, \cdots, k-1.\label{yi}
\end{gather}
These elements satisfy
\[
y_i T_l=T_l y_i  \qquad\text{for}\quad l\neq i, \ i-1.
\]

Furthermore, the elements $y_i$ can be written as follows:
\[
\begin{array}{ccc}
  y_i & = & 1+(v^{-1}-v)t^{-1}(T_{(1\ i)}+T_{(2\ i)}+\cdots +T_{(i-1\ i)}),
\end{array}
\]
where $T_{(m\ n)}$ belong to $H_k(v,t)$ associated to the transposition $(m\ n)\in S_k$.
In particular, $y_1, \cdots, y_k$ generate a commutative subalgebra of $H_k(v,t)$.

For any $i=1,\cdots, k$, we let $w_i$ denote the unique longest element
of the symmetric group $S_i$ which is regarded as the natural subgroup of $S_k$.
The corresponding elements $T_{w_i}\in H_k(v,t)$
are then given by $T_{w_1}=1$ and
\begin{equation}
\begin{aligned}
T_{w_i}&=T_1(T_2T_1)\cdots(T_{i-2}T_{i-3}\cdots T_1)(T_{i-1}T_{i-2}\cdots T_1)\\
&=(T_{1}\cdots T_{i-2}T_{i-1})(T_{1}\cdots T_{i-3}T_{i-2})\cdots (T_1T_2)T_1,\ i=2,\cdots, k.
\end{aligned}
\end{equation}

It is easily check that
\begin{equation}
\begin{aligned}
T_{w_i}T_j=T_{i-j}T_{w_i},\ \ 1\leq j< i\leq k,\\
T_{w_i}^2=t^{2(k-1)}y_1y_2\cdots y_i, \ \ i=1,\cdots,k.
\end{aligned}
\end{equation}

Following \cite{nazar}, for any $i=1,\cdots, k$, we define the elements:
\begin{equation}\label{Txy}
  T_i(x,y)=t^{-1}T_i+\frac{(v^{-1}-v)x}{y-x},
\end{equation}
where $x$ and $y$ are complex variables. We will regard the $T_i(x, y)$ as rational functions in $x$ and $y$ with values in $H_k(v,t)$.
These functions satisfy the braid  relations:
\begin{equation}
T_i(x,y)T_{i+1}(x,z)T_i(y,z)=T_{i+1}(y,z)T_i(x,z)T_{i+1}(x,y),
\end{equation}
and
\begin{equation}
T_i(x,y)T_i(y, x)=\frac{(x-v^{-2}y)(x-v^2y)}{(x-y)^2}.
\end{equation}

Following \cite{dj}, we will identify
a partition $\lambda=(\lambda_1,\dots,\lambda_l)$ of $k$
with its Young diagram which is
a left-justified array of rows of cells such that the first row
contains $\lambda_1$ cells, the second row contains
$\lambda_2$ cells, etc.
A cell $\tau$ outside $\lambda$ is called  addable to $\lambda$
if the union of the cell $\tau$ and $\lambda$ is a Young diagram.
A tableau $\mathcal{T}$ of shape $\lambda$ is obtained by
filling in the cells of the diagram
bijectively with the numbers $1, \cdots, k$.
A tableau $\mathcal{T}$ is called standard if its entries
increase along the rows and down the columns.
If a cell occupied by $i$ occurs in row $m$ and column $n$,
its $(v,v^{-1})$- content $\sigma_i$ will be defined as $v^{-2(n-m)}$.

In accordance to \cite{dj}, a set of primitive orthogonal
idempotents $\{E^{\lambda}_{\mathcal{T}}\}$ of $H_k(v,t)$, parameterized
by partitions $\lambda$ of $k$ and standard tableaux $\mathcal{T}$ of shape $\lambda$
can be constructed inductively as follows.
If $k=1$, set $E^{\lambda}_{\mathcal{T}}=1$. For $k\geq 2$, one defines inductively that
\begin{equation}\label{E}
  E_{\mathcal{T}}^{\lambda}=E_\mathcal{U}^{\mu}\frac{(y_k-\rho_1)\cdots(y_k-\rho_l)}
{(\sigma-\rho_1)\cdots(\sigma-\rho_l)},
\end{equation}
where $\mathcal{U}$ is the tableau of shape $\mu$ obtained form $\mathcal{T}$ by removing the cell $\alpha$ occupied by $k$,  and
$\rho_1,\cdots,\rho_l$ are the $(v,v^{-1})$-contents of all the addable cells
of $\mu$ except for $\alpha$, while $\sigma$ is the $(v,v^{-1})$-content of $\alpha$.

These elements become a family of primitive orthogonal idempotents of $H_k(v,t)$. Indeed,
if $\lambda$ and $\lambda'$ are distinct partitions of $k$, and $\mathcal{T}$ (respectively $\mathcal{T'}$) is any standard tableau of shape  $\lambda$ (respectively $\lambda'$), then we have
\begin{equation}
E_\mathcal{T}^{\lambda}E_{\mathcal{T'}}^{\lambda'}=\delta_{\lambda,\lambda'}
\delta_{\mathcal{T},\mathcal{T'}}E_\mathcal{T}^{\lambda}.
\end{equation}
 Moreover,
\begin{equation}
\sum_{\lambda}\sum_{\mathcal{T}}E_{\mathcal{T}}^{\lambda}=1,
\end{equation}
summed over all partitions $\lambda$ of $k$ and all the standard tableaux $\mathcal{T}$ of shape $\lambda$.

\begin{ex}\label{ex1}
For $k=2$, then $\lambda_1=(1,1)$ and $\lambda_2=(2)$ are all possible partitions of $k$. Set
$\mathcal{T}_1={\scriptsize\young(1,2)}$ (resp. $\mathcal{T}_2={\scriptsize\young(12)}$) be the only standard tableau of $\lambda_1$ (resp. $\lambda_2$).

Consider $E_{\mathcal{T}_1}^{\lambda_1}$. Obviously, $\mathcal{U}={\scriptsize\young(1)}$ is
 the tableau of shape $\mu=(1)$ obtained form $\mathcal{T}_1$ by removing the cell occupied by $2$,  and $\rho_1=v^{-2}$, $\sigma=v^2$. Then
\[
\begin{array}{ccl}
\medskip
  E_{\mathcal{T}_1}^{\lambda_1} & = & \frac{y_2-\rho_1}{\sigma-\rho_1} \\
\medskip
                                & = & \frac{t^{-2}T_1^2-v^{-2}}{v^2-v^{-2}} \\
\medskip
                                & = & -\frac{vt^{-1}}{1+v^2}T_1+\frac{1}{1+v^2}.
\end{array}
\]

Similarly, we can calculate $E_{\mathcal{T}_2}^{\lambda_2}=\frac{vt^{-1}}{1+v^2}T_1+\frac{v^2}{1+v^2}$.
As we see, the set of $E_{\mathcal{T}_1}^{\lambda_1}$ and $E_{\mathcal{T}_2}^{\lambda_2}$
is a complete set of primitive orthogonal idempotents of $H_2(v,t)$.
\end{ex}

\section{Fusion formulas for $E_\mathcal{T}^{\lambda}$ of $H_k(v,t)$}\label{sec fusion}

 We now apply the fusion formulas \cite{imo}  for the primitive orthogonal idempotents of two-parameter Hecke algebra $H_k(v,t)$.

Let $\lambda=(\lambda_1,\cdots,\lambda_l)$ be a partition of $k$, $\lambda'=(\lambda_1',\cdots,\lambda_{l'}')$ be the conjugate partition of
$\lambda$ obtained by turning the rows into columns.
If a cell $\alpha$ occurs in the $(i,j)$-th position of $\lambda$, denoted by $\alpha=(i,j)$, then the corresponding
hook is defined as $h_{\alpha}=\lambda_i+\lambda_j'-i-j+1$. Set
\begin{equation}\label{}
  f(\lambda)=v^{-b(\lambda)}t^{\frac{k(k-1)}{2}}(1-v^{-2})^k\prod_{\alpha\in \lambda}
(1-v^{-2h_\alpha})^{-1},
\end{equation}
where $b(\lambda)=\sum\limits_{i\geq 1}\lambda_i(\lambda_i-1)$, the sum is carried out all cells $\alpha$ of $\lambda$.

Now we introduce the rational function $\Psi(u_1, \cdots, u_k)$ in complex variables $u_1,\cdots, u_k$
with values in $H_k(v,t)$ by the following way:
\begin{equation}\label{}
  \Psi(u_1, \cdots, u_k)=\prod_{i=1,\cdots, k-1}^{\longrightarrow}
(T_i(u_1,u_{i+1})T_{i-1}(u_2,u_{i+1})\cdots T_1(u_i,u_{i+1}))\cdot
T_{w_k}^{-1},
\end{equation}
where the product is carried out in the order of $i=1, \cdots, k-1$.

\begin{prop}
For the partition $\lambda$ of $k$ and a standard tableau $\mathcal{T}$ of shape $\lambda$,
the primitive orthogonal idempotents $E^{\lambda}_\mathcal{T}$ can be obtained by the consecutive evaluations
\begin{equation}
E^{\lambda}_\mathcal{T}=f(\lambda)\Psi(u_1,...u_k)
|_{u_1=\sigma_1}|_{u_2=\sigma_2}\cdots|_{u_k=\sigma_k}.
\end{equation}
\end{prop}

\begin{ex}\label{ex2}
As example \ref{ex1}, we take $k=2$, $\lambda_2=(2)$. Then
\[
\begin{array}{ccl}
\medskip
  f(\lambda_2) & = & v^{-2}t(1-v^{-2})^2(1-v^{-2})^{-1}(1-v^{-4})^{-1}\\
\medskip
  & = &  \frac{t}{1+v^2},
\end{array}
\]
where $b(\lambda_2)=2$, $h_{(1,1)}=2$,  $h_{(1,2)}=1$.
Since $\sigma_1=1$, $\sigma_2=v^{-2}$, so
\[
\begin{array}{ccl}
\medskip
\Psi(u_1,u_2)|_{u_1=\sigma_1}|_{u_2=\sigma_2}&=&(t^{-1}T_1+\frac{v^{-1}-v}{v^{-2}-1})T_1^{-1}\\
\medskip
 &=& vt^{-2}T_1+v^2t^{-1}.
\end{array}
\]
Thus, the idempotent
\[
\begin{array}{ccl}
\medskip
E^{\lambda_2}_{\mathcal{T}_2}&=&f(\lambda_2)\Psi(u_1,u_2)|_{u_1=\sigma_1}|_{u_2=\sigma_2}\\
\medskip
&=& \frac{vt^{-1}}{1+v^2}T_1+\frac{v^2}{1+v^2}.
\end{array}
\]
The result coincide with example \ref{ex1}.
\end{ex}

Since $E_\mathcal{T}^{\lambda}$ is a primitive idempotent of $H_k(v,t)$,
$E_\mathcal{T}^{\lambda}$ acts on the simple module $V^{\lambda}$ of
$H_k(v,t)$ as a projector on a 1-dimensional subspace
and when $\lambda\neq \lambda'$, $E_\mathcal{T}^{\lambda}$ annihilates the irreducible
$H_k(v,t)$-module $V^{\lambda'}$. Furthermore, using Corollary \ref{decom},
we can get the following explicit description of the irreducible
modules of $U_{v, t}(sl_n)$.

\begin{thm}
For a partition $\lambda=(\lambda_1,...,\lambda_l)$ of $k$ with length $l\leq n$ and $\mathcal{T}$ a standard tableau of type $\lambda$,
then
$$V(\lambda)=E^{\lambda}_\mathcal{T}(V^{\otimes k})$$
is the finite dimensional irreducible representation of $U_{v, t}(sl_n)$.
\end{thm}

\begin{ex}
There are two partitions $\lambda_1=(1,1)$ and $\lambda_2=(2)$ for $k=2$, and both of their length  are not bigger than $n\geq 2$. Combine with the results in example \ref{ex1}, a computation shows that
\begin{eqnarray*}
   E^{\lambda_1}_{\mathcal{T}_1}(V^{\otimes 2}) &=& span_{\Q(v,t)}\{ v_i\otimes v_j-vtv_j\otimes v_i \mid 1\leq i < j\leq n \}, \\
  E^{\lambda_2}_{\mathcal{T}_2}(V^{\otimes 2}) &=& span_{\Q(v,t)}\{v_i\otimes v_i\mid 1\leq i \leq n\}\cup\{ v_i\otimes v_j+v^{-1}tv_j\otimes v_i \mid 1\leq i < j\leq n \},
\end{eqnarray*}
they  are precisely the irreducible $U_{v, t}(sl_n)$-submodules of $V\otimes V$.
\end{ex}

\noindent{Acknowledgements: This work is supported by NSFC 11571119 and NSFC 11475178.}

\end{document}